\documentclass[sn-nature]{sn-jnl}

\usepackage{natbib}

\usepackage{graphicx}%
\usepackage{multirow}%
\usepackage{amsmath,amssymb,amsfonts}%
\usepackage{amsthm}%
\usepackage{mathrsfs}%
\usepackage[title]{appendix}%
\usepackage{xcolor}%
\usepackage{textcomp}%
\usepackage{manyfoot}%
\usepackage{booktabs}%
\usepackage{algorithm}%
\usepackage{algorithmicx}%
\usepackage{algpseudocode}%
\usepackage{listings}%

\input cyracc.def

\usepackage{color}

\theoremstyle{plain}

\newtheorem{thm}{Theorem}[section]

\newtheorem{cor}[thm]{Corollary}
\newtheorem{prop}[thm]{Proposition}

\theoremstyle{definition}
\newtheorem{dfn}[thm]{Definition}

\definecolor{blue}{rgb}{0,0,0}

\def\dnfo{\;\raise.2em\hbox{$\mathrel|\kern-.9em\lower.4em\hbox
{$\smile$}$}}

\def\dnf#1{\lower.9em\hbox{$\buildrel\dnfo\over{ \scriptstyle  #1}$}}

\def\dfo{\;\raise.2em\hbox{$\mathrel|\kern-.9em\lower.4em\hbox{$\smile$}
\kern-.72em\lower.07em\hbox{\char'57}$}\;}
\def\df#1{\lower1em\hbox{$\buildrel\dfo\over{\scriptstyle #1}$}}

\newcommand{\bN}{\mathbb{N}}

\newcommand{\real}{\mathbb{R}}

\newcommand{\rat}{\mathbb{Q}}

\newcommand{\comp}{\mathrm{comp}}

\newcommand{\halt}{\downarrow}

\begin{document}

\title[Normality, relativization, and randomness]{ Normality, relativization, and randomness}


\author*[1]{\fnm{Wesley} \sur{Calvert}}\email{wcalvert@siu.edu}

\author[2]{\fnm{Emma} \sur{Gruner}}\email{eeg67@psu.edu}

\author[3]{\fnm{Elvira} \sur{Mayordomo}}\email{elvira@unizar.es}

\author[4]{\fnm{Daniel} \sur{Turetsky}}\email{dan.turetsky@vuw.ac.nz}

\author[5]{\fnm{Java Darleen} \sur{Villano}}\email{javavill@uconn.edu}

\affil*[1]{\orgdiv{School of Mathematical and Statistical Sciences}, \orgname{Southern Illinois University}, \orgaddress{\street{Mail Code 4408}, \city{Carbondale}, \postcode{62918}, \state{Illinois}}}

\affil[2]{\orgdiv{Department of Mathematics}, \orgname{Penn State University}, \orgaddress{\city{University Park}, \postcode{16802}, \state{Pennsylvania}}}

\affil[3]{\orgdiv{Departamento de Inform{\'a}tica e Ingenier{\'\i}a
de Sistemas, Instituto de Investigaci{\'o}n en Ingenier{\'\i}a de
Arag{\'o}n}, \orgname{Universidad de Zaragoza}, \country{Spain}}

\affil[4]{\orgdiv{School of Mathematics and Statistics}, \orgname{Victoria University of Wellington}, \city{Wellington},  \country{New Zealand}}

\affil[5]{\orgdiv{Department of Mathematics}, \orgname{University of Connecticut}, \orgaddress{\city{Storrs}, \postcode{06269}, \state{Connecticut}}}


\date{\today}



\abstract{
Normal numbers were introduced by Borel and later proven to be  a weak notion of algorithmic randomness.  We introduce here a natural  relativization of normality based on generalized number representation systems.  We explore the concepts of supernormal numbers that correspond to semicomputable relativizations, and that of highly normal numbers  in terms of computable ones.

We prove several properties of these new randomness concepts.
Both supernormality and high normality generalize Borel absolute normality. Supernormality is strictly between 2-randomness and effective dimension 1, while high normality corresponds exactly to sequences of computable dimension 1 providing a more natural characterization of this class.

}

\maketitle

\section{Introduction}

Normal numbers were introduced by Borel in \cite{Borel1909}, and used as an early model of randomness in \cite{Popper1935}.  Normality is certainly a weak notion of randomness since a number is $b$-normal if it is random for all finite-state gambling machines \cite{BoHiVi05}.
We say that a number $x$ is \emph{normal to base $b$} if and only if, in the base $b$ expansion of $x$, every block of digits of a given length occurs with the same limiting frequency.  This property is heavily sensitive to the base $b$; the same number may be normal to one base but not normal to another.  It is customary, then, to strengthen normality in this way: a number $x$ is said to be \emph{absolutely normal} if and only if it is normal to every base.  This is still a weak notion of randomness since examples of computable absolutely normal numbers are given in \cite{Turi13, Becher2002, LuMa21}.  A recent reference on normal numbers is \cite{Bugeaud2012}.

In the present paper, we introduce a natural relativization of normality through general number representation systems that extend the representation of real numbers in a fixed base. 
Certainly the standard representation of a real number $x$ in base $b$ as \[x = \sum\limits_{i\in \mathbb{Z}} a_i b^{i},\] where only finitely many of the $a_i$ with positive $i$ are nonzero, is a useful representation, but it is not the only one imaginable, and it is, after all, a standard observation that an appropriate choice of data structure can radically change the complexity of a computational problem.  We consider, then, alternate representations of real numbers.  Let $\Sigma_b=\{0, \ldots, b-1\}$.  A system for representing real numbers is given by a function $f:\Sigma_b^{<\omega} \to \rat$, where the domain is interpreted as strings naming a number and the codomain is interpreted as the represented rational numbers.  A real number $x$ is named by a sequence of strings $(\sigma_i : i \in \mathbb{N})$ such that $\lim\limits_{i \to \infty} f(\sigma_i) = x$.  By varying the effectiveness of $f$ we can then obtain alternative concepts to Borel normality.

We explore here the cases of semicomputable and computable number representation systems obtaining
the new concepts of supernormal and  highly normal numbers.  Just as absolutely normal numbers are numbers whose normality is robust to change of base, these numbers are those whose normality is, in different senses, robust to all reasonable changes of representation.

We compare supernormality and high normality with existing randomness notions, showing that supernormal numbers are strictly between 2-randomness and effective dimension 1, and that 
highly normal numbers are exactly those of computable dimension 1, which can be a more natural characterization of this interesting class.

\subsection{Preliminaries}


Fix $U$ a  universal {\color{blue} prefix-free} Turing machine and let $K$ be the {\color{blue} prefix-free} Kolmogorov complexity defined for each $w\in 2^{<\omega}$ by \[K(w) =\min\{|p| : U(p)=w\}.\]

For other {\color{blue} prefix-free} Turing machines $M$, \[K_M(w) =\min\{|p| : M(p)=w \ \mathrm{or}\ p=w \}.\]

{\color{blue} Notice that we can relatize $K$ (and $K_M$) by allowing the (oracle) Turing machine to access an oracle $A$. that is, $K^A(w) =\min\{|p| : U^A(p)=w\}$.}


{\color{blue}

For each} $x\in [0,1]$, we identify $x$ with the binary sequence in $2^{\omega}$ corresponding to the binary representation of $x$. {\color{blue} We denote with $\overline{x}$ or the {\sl complement of $x$\/} the real number with binary representation complementary to that of $x$, that is, $\overline{x}=1-x$.  }

We next define the {\sl effective dimension\/} of a number.
\begin{dfn}\cite{KCCCHD} Let $x\in 2^{\omega}$. The \textit{dimension} of $x$ is
\[\dim(x)=\liminf_n \frac{K(x[1..n])}{n}.\]
\end{dfn}


{\color{blue} $n$-dimension is the relativization of the last concept.

\begin{dfn} Let $x\in 2^{\omega}$, $n\in\bN$. The \textit{$n$-dimension} of $x$ is
\[n-\dim(x)=\liminf_i \frac{K^{\emptyset^{n-1}}(x[1..i])}{i}.\]
\end{dfn}

Lutz introduced computable dimension in \cite{DCC}. We use here the definition below following from \cite{Hitch05}\ and that is also in terms of Kolmogorov complexity. 
}

\begin{dfn} Let $x\in 2^{\omega}$. The \textit{computable dimension} of $x$ is
\[\dim_{\comp}(x)=\inf\limits_{M}\left\{\liminf_n \frac{K_M(x[1..n])}{n}\right\},\] where $M$ ranges over all Turing machines which halt on every natural number.
\end{dfn}

{\color{blue}

Let $x\in [0,1]$, let $b\in\bN$, we denote as $seq_b(x)$  the base-$b$ representation of $x$.  For  $\sigma\in \Sigma_b^{<\omega}$ we identify $\sigma$ with the rational  $0.\sigma$ that has $\sigma$ as its finite representation in base $b$.

For the sake of completeness we include here the definition of Borel normal number.

\begin{dfn}[\cite{Borel1909}] Let $x\in[0,1]$, let $b\in\bN$. $x$ is {\sl $b$-normal\/} if for every $m\in\bN$ and every $w\in \Sigma_b^{m}$, the asymptotic, empirical frequency of $w$ in  $seq_b(x)$  is $b^{-m}$, that is
\[\lim_n \frac{|\{i\le n\,|\,seq_b(x)[i-m+1..i]=w\}}{n}= b^{-m}.\]
 $x$ is {\sl absolutely normal\/} if $x$ is {$b$-normal} for every $b\in\bN$.

\end{dfn}

{\color{blue}

We also include the definition of $n$-randomness. 

Let $\mu$ be Lebesgue measure on $2^{\omega}$. For $A\subseteq 2^{<\omega}$ we denote with $\mu(A)=\mu(\{x \,|\, \exists i \ x[1..i]\in A\})$. 

\begin{dfn} $x\in[0,1]$, $n\in\bN$. $x$ is $n$-random if for every  $U_k\subseteq 2^{<\omega}$ that is uniformly computably enumerable relative to $\emptyset^{n-1}$ and such that  for every $k$ 
$\mu\left(U_{k}\right) < 2^{-k}$, $x$ is not covered by $U_k$, that is, there is a $k_0$ such that for all $i$ $x[1..i]\not\in U_{k_0}$.

\end{dfn}

Finally we recall the definition of Chaitin's omega.

\begin{dfn}  Let $U$ be a universal Turing machine. Then Chaitin's omega is $\Omega_U =
\sum\limits_{U(\sigma)\halt} 2^{-|\sigma|}$.
When $U$ is fixed we use the notation $\Omega$.

\end{dfn}

}

}


\subsection{Finite-state characterizations of Borel normality}


\begin{dfn}  
A \textit{finite state (FS) $\Sigma_b$-transducer} $D$ is a tuple $D=(Q,q_0, \delta, out)$ such that $Q$ is a finite set of states, $q_0\in Q$ is the initial state, $\delta: Q\times  \Sigma_b \to  Q$ is the transition function, and $out: Q\times \Sigma_b \to  \Sigma_b^{<\omega}$  is the output function. 
\end{dfn}

There is a natural notion of (plain) Kolmogorov complexity with respect to a transducer.

\begin{dfn} Let $b\in\bN$, let $D$
be a FS $\Sigma_b$-transducer, $\sigma\in \Sigma_b^{<\omega}$ (which is the finite representation in base $b$ of the rational denoted $0.\sigma$), $n\in\bN$, and $x\in [0,1]$, \[C_{D}(\sigma) := \min\{|p|: D(p)=\sigma\}\]
\[C_{n,D}(x) := \min\{C_D(\sigma): |0.\sigma-x| < b^{-n}\}\cup\{n+1\}.\]
Notice that if $\sigma$ is not a $D$-output, then $C_{D}(\sigma)=\infty$.
\end{dfn}


{\color{blue}
\begin{dfn} Let $b\in\bN$, a {\sl $\Sigma_b$-martingale\/} $d$ is a function $d: \Sigma_b^{<\omega}\to [0, \infty)$ such that for every $w\in \Sigma_b^{<\omega}$,
\[d(w) = \frac{\sum_{a\in\Sigma_b} d(wa)} {b^{|w|}}\]
where $|w|$ is the length of $w$.

A {\sl FS-computable $\Sigma_b$-martingale\/} $d$ is a $\Sigma_b$-martingale such that there is a FS-transducer that on every input $w$ outputs $d(w)$ (via the identification of  $\sigma\in \Sigma_b^{<\omega}$ with the rational  $0.\sigma$).

\end{dfn}}

Borel normality can be characterized both in terms of finite-state martingale success and in terms of  finite-state compressibility. While in terms of martingales randomness (that is, no infinite success) and dimension one (that is, no linear exponential success) coincide, the same is probably false for the case of compressibility (for instance, for the  Lempel-Ziv algorithm, which compresses more efficiently than any finite state machine, randomness or maximal incompressibility is different from dimension one or compressibility ratio 1).

\begin{thm}[Characterization Theorem]\label{CharacterizationTheorem} Let $x\in [0,1]$, let $b\in\bN$, let $h:\mathbb{N}\to\mathbb{N}$ be such that $\lim_n h(n)=\infty$ and $h(n)=o(n)$. 
The following are equivalent:
\begin{enumerate}
\item $x$ is normal to a base $b$.
\item For each FS-computable $\Sigma_b$-martingale $d$, for every $\epsilon>0$ and almost every $n$, we have \[d(seq_b(x)[1..n])<2^{\epsilon n}.\]
\item For each FS-computable $\Sigma_b$-martingale $d$, there is an $\epsilon>0$ and $c>0$ such that for almost every $n$,  we have either \[d(seq_b(x)[1..n])<2^{-\epsilon n}\] or else \[d(seq_b(x)[1..n])=c.\] 
\item For each FS-computable $\Sigma_b$-martingale $d$ and almost every $n$,  we have \[d(seq_b(x)[1..n])<2^{h(n)}.\]
\item For each FS $\Sigma_b$-transducer $D$, for every $\epsilon>0$, and for almost every $n$, we have \[C_{D}(seq_b(x)[1..n])>n (1-\epsilon).\]
\item For each FS $\Sigma_b$-transducer $D$, for every $\epsilon>0$, and for almost every $n$, we have \[C_{n,D}(x)>n (1-\epsilon).\]
\end{enumerate}
\end{thm}

\begin{proof}
{\color{blue}Proposition 4.1 (Satz 4.1) in \cite{SchSti72}\  is the celebrated Schnorr-Stimm dichotomy theorem of finite-state martingales, stating that for each infinite sequence $z$, if $z$ is normal then for each FS-computable martingale  $d$ there is an $\epsilon>0$ and $c>0$ such that for almost every $n$,  $d(z[1..n])<2^{-\epsilon n}$ or else $d(z[1..n])=c$; and if $z$ is not normal then there is 
a FS-computable martingale $d$ and $\epsilon>0$ such that 
$d(z[1..n])>2^{(1+\epsilon) n}$ for infinitely many $n$.

Using Proposition 4.1 in \cite{SchSti72}\ we have that (1)-(4) are all equivalent.

\cite{FSD}\ proves, for any sequence $z$, the equivalence of (2)  (for each FS-computable $\Sigma_b$-martingale $d$, for every $\epsilon>0$ and almost every $n$, $d(z[1..n])<2^{\epsilon n}$) with the property of $z$ being incompressible by finite-state information-lossless compressors.

\cite{DotMos06}\ proves the equivalence of $z$ being incompressible by finite-state information-lossless compressors with (5), that is, for each FS $\Sigma_b$-transducer $D$, for every $\epsilon>0$, and for almost every $n$, $C_{D}(z[1..n])>n (1-\epsilon)$.

Therefore (2) is equivalent to (5).

Finally \cite{PSPFSD}\ proves the equivalence of (5) and (6).

}
\end{proof}

\subsection{Relativizing Normality to Other Representation Systems}

{\color{blue} In order to relativize normality, let us start by reflecting on the oracle use required in the definition of relativized effective dimension, $\dim^A(x)$. The following result  characterizes dimension in terms of Kolmogorov complexity at a certain precision.

\begin{dfn}[\cite{DPSSF}] Let $x\in [0, 1]$, $r\in\bN$. 
\[ K_r(x)= \min\{|p| : |U(p)-x|<2^{-r}\}.\]
\end{dfn}

\begin{thm}[\cite{DPSSF}] Let $x\in[0,1]$. Then
\[\dim(x)=\liminf_r \frac{K_r(x)}{r}.\]
\end{thm}
Now let us think about  $\dim^A(x)$, which is based on $K_r^A(x)$, that is, on the enumeration of real numbers given by $U^A(p)$ as $p$ ranges over $2^{\omega}$. The main role of the oracle $A$ in $K_r^A$ is to provide a representation system for real numbers in  $[0, 1]$ through this enumeration, $f(p)= U^A(p)$.
}

While the standard definition of normality given above is not naturally relativizable,  Theorem \ref{CharacterizationTheorem} gives us several other options.  We first describe how  Kolmogorov complexity can be relativized to a representation system.

\begin{dfn} Let $b\in\bN$, $f: \Sigma_b^{<\omega}\to \rat $, $D$
be a FS $\Sigma_b$-transducer,  $n\in\bN$, and $x\in [0,1]$. Then we define $C_n^f, C_{n,D}^f:[0,1]\to\mathbb{N}$ by
\[C^f_{n}(x) := \min(\{|\sigma|: |f(\sigma)-x| < b^{-n}\}\cup\{n+1\}),\]
\[C^f_{n,D}(x) := \min(\{C_D(\sigma): |f(\sigma)-x| < b^{-n}\}\cup\{n+1\}).\]
\end{dfn}

We now apply Clauses 5 and 6 of Theorem \ref{CharacterizationTheorem} to propose a relativized notion of normality.

\begin{dfn} Given a function $f: \Sigma_b^{<\omega}\to \rat$, we say that a real $x$ is \textit{strongly $f$-normal} if and only if for every finite state machine $D$ and every $\epsilon>0$, for almost every $n$, we have  \[C^f_{n,D}(x) \geq n(1-\epsilon).\] \end{dfn}

A minimum check on such a definition is that a real which is strongly $f$-normal for some appropriate choice of $f$ should be normal.  Indeed, we can let $f_b$ be the function from $\Sigma_b^{<\omega}$ to $[0,1]$ which returns the numerical value of a string in $\Sigma_b^{<\omega}$ (under the usual $b$-ary representations). Then Theorem \ref{CharacterizationTheorem} says exactly that every real which is strongly $f_b$-normal is normal to base $b$.

We could also imagine another formulation omitting the transducer $D$.  

\begin{dfn} Given a  function $f$, we say that a real $x$ is \textit{weakly $f$-normal} if and only if for every $\epsilon>0$, for almost every $n$, we have  \[C^f_{n}(x) \geq n(1-\epsilon).\]\end{dfn}

As with the ``strong" definition, we would hope that every normal number is weakly $f$-normal for an appropriate function $f$.  Using the {\color{blue} usual base-$b$} representation $f_b$ as before, Theorem \ref{CharacterizationTheorem} shows that every real which is normal to base $b$ is weakly $f_b$-normal.

\begin{prop} For every function $f: \Sigma_b^{<\omega}\to \rat$, every real number $x$ which is strongly $f$-normal must also be weakly $f$-normal.  On the other hand, there is a function $f$ and a real $x$ such that $x$ is weakly $f$-normal but not strongly $f$-normal.\end{prop}

\begin{proof}  Let $x$ be strongly $f$-normal.  Then, in particular, if $D$ is the identity transducer, we have $C^f_{n,D} = C^f_n$, so that $x$ is weakly $f$-normal.  On the other hand, 
 let $b>2$, and $f$ be the function which returns the numerical value of a string in $\Sigma_b^{<\omega}$.  Let $x=\sum\limits_{i\in \mathbb{N}} (b-2)b^{-i}$. Then $C^f_{n}(x)=n$, while $C^f_{n,D}(x) =n/2$ for any finite automaton $D$ that, on input $(b-2)$ outputs the concatenation $(b-2)(b-2)$.    
\end{proof}

Notice that for every $x \in \real$ there is a function $f$ such that $x$ is not weakly $f$-normal.
We believe that the study of weak and strong $f$-normality can be of independent interest, for instance in terms of equidistribution properties \cite{PSPFSD}. 

In the following two sections, we {\color{blue}describe} the properties of numbers which are (strongly or weakly) $f$-normal with respect to a large family of functions $f$.  In Section \ref{sec:supernormal} we consider the family of upper semi-computable functions (so that the effectiveness condition itself includes an element of approximation), and in Section \ref{sec:highlynormal} we consider the family of computable functions.

\section{Supernormal Numbers}\label{sec:supernormal}

It is reasonable to expect that an effective notion of approximation {\color{blue} allows} for approximation in the effective functions themselves.  The following definition is common in situations of computation where approximation is called for.

\begin{dfn} A function $f: \Sigma_b^{<\omega}\to \rat$ is said to be \emph{upper semi-computable} if and only if there is a computable function $g: \mathbb{N}\times  \Sigma_b^{<\omega}\to \rat$ such that for each pair $(m,x)$, we have 
\begin{enumerate}
    \item $g(m,x) \geq f(x)$, 
    \item $g(m,x) \geq g(m+1,x)$, and
    \item $\lim\limits_{n \to \infty} g(n,x) = f(x)$.\end{enumerate}
    \end{dfn}

These functions have also been called \emph{approximable from above} \cite{CCM2013}. Notice that they are total functions.
{\color{blue} A \emph{universal upper semi-computable} function $U:2^{<\omega} \to 2^{<\omega}$ is such that for every $b\in\bN$ and every  upper semi-computable $f: \Sigma_b^{<\omega}\to \rat$ there is a $p$ such that for all $x$, we have $f(x)=U(\langle p,x\rangle )$. We now define a notion of descriptive complexity with respect to all upper semi-computable functions.}

\begin{dfn} We say that a real $x$ is \textit{supernormal} if and only if it is strongly $f$-normal with respect to all $b\in\bN$ and all upper semi-computable $f:\Sigma_b^{<\omega}\to \rat$.\end{dfn}

{\color{blue} In fact, for $f$ an upper semi-computable function, notions of descriptive complexity with or without finite-state machines coincide.}

\begin{thm}[]\label{raro1}The following properties of a real $x$ are equivalent.
\begin{enumerate}
    \item $x$ is supernormal.
    \item For all $b\in\bN$ and all upper semi-computable functions $f:\Sigma_b^{<\omega}\to \rat$, the real $x$ is weakly $f$-normal.
    \item The real $x$ is strongly $f$-normal for some universal upper semi-computable function $f$.
    \item The real $x$ is weakly $f$-normal for some universal upper semi-computable function $f$.
\end{enumerate}
\end{thm}
\begin{proof}
    
    We first show that ($1$) implies ($2$). Suppose for all upper semi-computable functions $f$, the real $x$ is strongly $f$-normal, i.e., for all finite state machines $D$ and for all $\epsilon$ and for almost all $n$,
    \[
    C_{n,D}^f(x)\geq (1-\epsilon)n.
    \]
    In particular, this holds for the identity automata $I$ where for all $\sigma$, we have $I(\sigma)=\sigma$, and $C_I(\sigma)=|\sigma|$. So, for an arbitrary upper semi-computable function $g$ and for $I$, we have that
    \[
        C_{n,I}^g(x)\geq (1-\epsilon)n,
    \]
    and moreover, $C_{n,I}^g(x)=C_n^g(x)$. Hence, $C_n^g(x)\geq (1-\epsilon)n$ for all $\epsilon$ and for almost all $n$.

    For ($2$) implies ($3$), suppose that for all universal upper semi-computable functions $f$, $x$ is not strongly $f$-normal, and so there exists a finite state machine $D$ and $\epsilon$, and there exist infinitely many $n$ where
    \[
    C_{n,D}^f(x) < (1-\epsilon)n.
    \]

    We define a new upper semi-computable function $g$ by letting $g(\tau)=f\circ D(\tau)$ for all $\tau\in 2^{<\omega}$ for some universal upper semi-computable $f$. Since $f$ was upper semi-computable, so is $g$. Then, we have
    \begin{align*}
        C_n^g(x) &= \min\{|\tau| : |g(\tau)-x|<2^{-n}\}\cup\{n+1\} \\
        &= \min\{C_D(D(\tau)) : |f(D(\tau))-x|<2^{-n}\}\cup\{n+1\}  \\
        &=\min\{C_D(\sigma) : |f(\sigma)-x|<2^{-n}\}\cup\{n+1\} = C_{n,D}^f(x).
    \end{align*}
    Since $C_{n,D}^f(x)< (1-\epsilon)n$, then $C_n^g(x)<(1-\epsilon)n$, and so $x$ is not weakly $g$-normal for our particular $g$.

    For ($3$) implies ($4$), suppose $x$ is strongly $f$-normal for some universal upper semi-computable function $f$. Then for the identity automata $I$, we have that for all $\epsilon$ and for almost all $n$
    \[
    C_{n,I}^f(x)\geq (1-\epsilon)n.
    \]
    By the definition of $I$, $C_I(\sigma)=|\sigma|$ and so $C_{n,I}^f(x)=C_n^f$. Hence, $C_n^f\geq (1-\epsilon)n$, and so $x$ is weakly $f$-normal.

    Finally we prove ($4$) implies ($1$). Suppose there exists some upper semi-computable $f$ such that $x$ is not strongly $f$-normal, and so there exists a finite state machine $D$ and an $\epsilon$ such that for all $k$ there is an $n$ where
    \[
    C_{n,D}^f(x) < (1-\epsilon)n-k.
    \]
   
   In particular, we have that $C_{n,D}^f(x) < (1-\epsilon)n-(k+c)$ for this finite state machine $D$ and $\epsilon$ and $n$ depending on $k+c$ with $k$ arbitrary. We now define $g=f\circ D$, and this is an upper semi-computable function since $f$ was upper semi-computable. Let $U$ be a universal upper semi-computable function, let $c$ be such that  for all $m$ and $y$, we have that
    \[
    C^U_m(y)\leq C^g_m(y)+c,
    \]
    We also get that $C^g_n(x)\leq C^f_{n,D}(x)<(1-\epsilon)n-(k+c)$. Therefore,
    \begin{align*}
    C^U_n(x) &\leq C^g_n(x)+c \\ &<(1-\epsilon)n-(k+c)+c \\
    &= (1-\epsilon)n-k.
    \end{align*}

    So $x$ cannot be weakly $U$-normal for given universal upper semi-computable function $U$.


\end{proof}

{\color{blue} 
We have chosen the class of upper semi-computable functions here. Notice that if, instead, we used the class of lower semi-computable functions, we would get the set of complements of the numbers we define as supernormal here. In this sense, our definition is robust up to complement or reflection.

}

We  next {\color{blue} explore} the relationship of supernormality with other algorithmic randomness notions.

\begin{prop} Every $2$-random real is supernormal.\end{prop}

\begin{proof}
    Suppose $x$ is not weakly $f$-normal, with $f$ upper semi-computable, then 
    there exists  $\epsilon$, such that for each $k$ there is an $n$ with
    $    C_{n}^f(x) < (1-\epsilon)n-k$.
 Since $f$ is upper semi-computable, we have $f \leq_T \emptyset'$.  
 Now for each constant $k$, we now describe a sequence of uniformly $\Sigma^{0,\emptyset'}_1$ classes $\left(U_{f,k,n} :n \in \mathbb{N}\right)$ such that $x \in \bigcap\limits_{n,k\in \mathbb{N}} U_{f,k,n}$.

    We let $U^0_{f,k,n}$ be the set of reals with the property that $y \in U^0_{f,k,n}$ if and only if there is a string $\sigma$ of length at most $(1-\epsilon)n-k$ with $|f(\sigma)-y|<2^{-2n}$.  Since there are only finitely many strings of that length, this condition can be checked by a $\emptyset'$ oracle (which is needed only to compute $f(\sigma)$).  We refine this sequence so that $U_{f,n,k}$ is defined to be $U^0_{f,m,k}$ for the least $m\geq n$ such that $\mu\left(U^0_{f,m,k}\right) < 2^{-n}$.
    \end{proof}

Indeed, something stronger is true.

\begin{prop} If $\dim^{\emptyset'}(x)=1$ (usually written $2$-$\dim(x)=1$) then $x$ is supernormal.
\end{prop}

 \begin{proof}
      Suppose $x$ is not weakly $f$-normal, with $f$ upper semi-computable, then 
    there exists  $\epsilon$, such that for infinitely many $n$,
    $    C_{n}^f(x) < (1-\epsilon)n$.
  So there is a $\sigma$ such that $|\sigma|<(1-\epsilon)n$ where $|f(\sigma)-x|< 2^{-n}$. 
   Since $f$ is upper semi-computable, we have $f \leq_T \emptyset'$ and  $C^{\emptyset'}(x[1..n])<(1-\epsilon)n+c+2+2\log n$ where $c$ is the length of a program computing $f$ with oracle $\emptyset'$, and we use two extra bits to decide on $f(\sigma)[1..n]$, its successor, or its predecessor. 
It follows that $\dim^{\emptyset'}(x)<1$.
 \end{proof}


\begin{prop}\label{snisrandom} If a real $x$ is supernormal, then $\dim(x)=1$.\end{prop} 
\begin{proof}
Let $U$ be a universal Turing machine. For $n\in\bN$, let $U_n$ be $U$ with time limit $n$.
We define an upper semi-computable function $f$ in the following way.  We set \[g(x,n)=\left\{ \begin{array}{ll} 1 & \mbox{if $U_n(x)\uparrow$}\\
U(x) & \mbox{if $U_n(x)\downarrow$}\end{array}\right.\]
and set $f(x) = \lim\limits_{n \to \infty} g(x,n)$.

Now suppose that $\dim(x)<1$. Then there is an $\epsilon$ where for infinitely many $n$, $C(x[1..n]) < (1-\epsilon)n$, and it follows that $C_n^f(x) < (1-\epsilon)n$ for infinitely many $n$.  Thus $x$ is not supernormal.
\end{proof}


Now we can show that there is a supernormal number which is not 1-random.  Indeed, if we let $x$ be $2$-random and for each $k$, insert the string $0^{k^2}$ after the first $2^k$ bits, the resulting $\tilde{x}$ infinitely often has $K(\tilde{x}[1..n])<n-\log^2 n$ and $\tilde{x}$ is not 1-random, but $2$-$\dim(\tilde{x})=1$, so $\tilde{x}$ is supernormal.

We now show that there is a 1-random real which is not supernormal.

\begin{prop}\label{cOmeganotSN} $\overline{\Omega}$ is not supernormal.\end{prop}

\begin{proof}
    We define $f$ in the following way.  Let $(q_i : i \in \mathbb{N})$ be a computable decreasing sequence of rationals converging to $\overline{\Omega}$. For each $x$, we define $g(x,0) = 1$. We now define, for $n>0$, \[g(\sigma,n) = \left\{\begin{array}{ll} q_n & \mbox{ if $|g(\sigma,n-1)-q_{n}|>2^{-2^{|\sigma|}}$}\\ g(x,n-1) & \mbox{otherwise}\end{array}\right. ,\] and take $f(\sigma) = \lim\limits_{n \to \infty} g(\sigma,n)$.

    Then we argue that $|f(\sigma)-\overline{\Omega}|\le 2^{-2^{|\sigma|}}$ and $\overline{\Omega}$ is not weakly  $f$-normal.
\end{proof}

\begin{prop}
    $\Omega$ is supernormal.
\end{prop}

\begin{proof}
We first {\color{blue} construct} a universal machine $U$ such that $\Omega_U$ is supernormal.  It {\color{blue} follows} then that $\Omega_W$ is supernormal for any universal machine $W$.

 Fix an upper semi-computable $f$.  We {\color{blue} show} that $\Omega$ is weakly $f$-normal, that is, that for all $\epsilon$ there is a $k$ such that for all $n$, $C^f_n(\Omega) \geq (1-\epsilon) n- k$. Let  $\Omega_s =
\sum\limits_{U_s(\sigma)\halt} 2^{-|\sigma|}$. For each $s$, let $f_s(\sigma)=f(\sigma, s)$. We show that if $n<s$ and $\sigma$ witnesses that  $C^{f_s}_n(\Omega_s) < (1-\epsilon) n- k \le n- K_s(n) - k$ then $|\Omega_s - \Omega| \ge 2^{-n+1}$ which implies that $f(\sigma)\le f(\sigma, s)\le \Omega- 2^{-n}$ and  $\sigma$ does not witness that $C^{f}_n(\Omega) < n- K_s(n) - k$.




    Recall that $\sum\limits_n 2^{-K(n)} < 1$.  Fix a $k$ such that $2^{-k}\sum\limits_n 2^{-K(n)} < \epsilon/2$.  At stage $s$, we search for an $n < s$ and a $\sigma$ with:
    \begin{enumerate}
        \item $0<| \Omega_s - f(\sigma, s)| < 2^{-n}$; and
        \item $|\sigma| < n - K_s(n) - k$.
    \end{enumerate}
    We increase $\Omega$ by the minimal amount necessary ($2^{-n+1}$) to arrange that $|\Omega - f(\sigma, s)| \ge 2^{-n}$ for each such $\sigma$ and $n$. Therefore
    $|\Omega - f(\sigma)| \ge 2^{-n}$.

    Notice that we may need to address a given $\sigma$ multiple times as $K$ drops: at stage $s$, $\sigma$ may require attention for $n+1$, but not for $n$ because $|\sigma| \ge n - K_s(n) - k$.  If at a stage $t > s$ we have $K_t(n) < K_s(n)$, then $\sigma$ may require attention for $n$ at stage $t$.  However, once we have addressed $\sigma$ for a given $n$, it will never again require attention for that $n$ (this is because the approximation to $f$ is non-increasing at $n$, while the approximation to $\Omega$ is non-decreasing).  In particular, it will not witness $C_n^f(\Omega) < n - K(n) - k$.
   
    If we fix the smallest $n$ with $|\sigma| < n - K(n) - k$, then we will issue at most a $2\cdot 2^{-n}$ increase to $\Omega$ on behalf of $\sigma$.  
 For this, we construct a Turing machine $M$ with coding constant $e$, that is, if $M(\rho)\halt$ then $U(\tau)\halt$ for $|\tau|=|\rho|+e$. For $\sigma$ we add $\rho$ with $\sigma\sqsubseteq \rho$ and $|\rho|=n-1-e$ to the domain of $M$, and therefore $2^{-n+1}$ to $\Omega$.
    
    There are fewer than $2^{n-K(n)-k}$ strings of length strictly less than $n-K(n) - k$, so if we partition the strings by $n$, we see that the total increase we issue is bounded by:
    \[
        \sum_{n} 2^{n-K(n)-k}\cdot 2 \cdot 2^{-n} = \sum_{n} 2^{-K(n)-k+1} < \epsilon.
    \]
    Thus we do not issue too much increase to $\Omega$, and the construction can proceed as described.  As we ensure that no $\sigma$ witnesses a failure of weak $f$-normality, $\Omega$ is weakly $f$-normal.
\end{proof}

\begin{cor}
    The class of supernormal numbers is strictly between the classes of 2-random and that  of dimension 1 numbers.
\end{cor}
\begin{proof}
    Notice that $\Omega$ is not 2-random (and in fact $\dim^{\emptyset'}(\Omega)=0$) and that $\dim(\overline{\Omega})=1$.
\end{proof}

The following result explains the connection of supernormality to absolute normality.

\begin{cor} If a real number $x$ is supernormal, then it is absolutely normal, and there exists an $x$ which is absolutely normal but not supernormal.\end{cor}

\begin{proof} 
This is an easy consequence of  Proposition \ref{snisrandom}. Additionally, \cite{Turi13}\ gives an example of an absolutely normal number which is computable, hence not supernormal.
\end{proof}

We now prove some basic properties of the class of supernormal numbers.

\begin{prop}\label{mult}
    If $x$ is supernormal and $q \in \rat^{+}$, then $qx$ is supernormal.
\end{prop}
\begin{proof}  
{\color{blue} Let  $f: \Sigma_b^{<\omega}\to \rat$ be upper semi-computable, $q \in \rat^{+}$. Then 
$qf(\sigma)= q\cdot f(\sigma)$ and $\frac{1}{q} f(\sigma)= \frac{1}{q}\cdot f(\sigma)$ are both upper semi-computable.

 If $x$ is weakly $f$-normal,} then $qx$ {\color{blue} is} weakly $qf$-normal.
 Moreover, if $qx$ fails to be weakly $f$-normal, then $x$  {\color{blue} fails} to be weakly $\frac{1}{q}f$-normal.   
\end{proof}

\begin{prop} There are supernormal numbers $x,y$ such that $x+y$ is approximable from above and hence not supernormal.\end{prop}

\begin{proof}
    Take some 2-random real $x$.  Then $-x$ is also 2-random, and thus they are both supernormal.  Now $x + (-x) = 0$ is certainly approximable from above.
\end{proof}

\begin{prop} There are $x,y$, each not supernormal, such that $x+y$ is supernormal.\end{prop}

\begin{proof} Let $z = \sum\limits_{i \in \mathbb{N}} z_i 2^{-i}$ be supernormal, let \[x = \sum\limits_{i \in \mathbb{N}} z_{2i}2^{-2i},\] and \[y = \sum\limits_{i \in \mathbb{N}} z_{2i+1}2^{-(2i+1)}.\]  Since $x$ and $y$ are not normal, they are not supernormal.\end{proof}

\begin{prop}
    Supernormal numbers are not closed under limits, i.e., there exists a sequence $(n_k)_{k\in\omega}$ of supernormal numbers such that $\lim\limits_{k\to\infty} n_k=n$ is not supernormal. 
\end{prop}
\begin{proof}
    Let $x$ be a supernormal, then define the sequence $\{n_k\}_{k\in\mathbb{N}}$ by
    $n_k = x2^{-k}$. Each $n_k$ is supernormal by Proposition \ref{mult}, but $\lim\limits_{k\to\infty} n_k=0$ and $0$ is not supernormal.
\end{proof}


\section{Highly Normal Numbers}\label{sec:highlynormal}



We now turn to our attention to a different class of numbers, the highly normal numbers, which are obtained when we require weak (and strong) $f$-normality for total computable functions $f$ as opposed to upper semi-computable functions.

\begin{dfn} We say that a real $x$ is \textit{highly normal} if and only if it is strongly $f$-normal with respect to all total computable $f$.\end{dfn}

{\color{blue} For highly normal numbers, we have a similar result to Theorem \ref{raro1}, though there are no conditions analogous to ($3$) and ($4$) from Theorem \ref{raro1} mentioning a universal total computable function $f$ because one does not exist.}

\begin{thm}\label{highlynormequiv}The following properties of a real $x$ are equivalent.
\begin{enumerate}
    \item For all total computable functions $f$, the real $x$ is strongly $f$-normal.
    \item For all total computable functions $f$, the real $x$ is weakly $f$-normal.
\end{enumerate}
\end{thm}



\begin{proof}
  
    We first show that ($1$) implies ($2$). Suppose for all computable functions $f$, the real $x$ is strongly $f$-normal, i.e., for all finite state machines $D$ and $\epsilon$ and for almost all $n$,
    \[
    C_{n,D}^f(x)\geq (1-\epsilon)n.
    \]
    So, in particular, this holds for the identity automata $I$ where $C_I(\sigma)=|\sigma|$. So, for an arbitrary total computable function $g$ and for $I$,
    \[
        C_{n,I}^g(x)\geq (1-\epsilon)n,
    \]
    and moreover, $C_{n,I}^g(x)=C_n^g(x)$. Hence, $G_n^g(x)\geq (1-\epsilon)n$, and since $g$ was arbitrary, we have shown ($2$).


    Suppose $x$ is not strongly $f$-normal for some computable function $f$. So, there exists a finite state machine $D$ and $\epsilon$ such that there exists an $n$ where
    \[
    C_{n,D}^f(x)<(1-\epsilon)n.
    \]

    Let $g=f\circ D$, then because $D$ is a finite state machine and $f$ is computable, $g$ is a computable function. By the same argument for ($2$) implies ($3$) in Theorem \ref{raro1}, we have that $C_n^g(x)\leq C_{n,D}^f$, and so $C_n^g(x)<(1-\epsilon)n$ and thus $x$ is not weakly $g$-normal for some computable function $g$.

\end{proof}

As before, we note that every highly normal number is absolutely normal. Moreover, every supernormal is highly normal, and this implication is strict.


The following theorem is important in that it gives a novel characterization of the reals with computable dimension 1.  To our taste, the definition of highly normal numbers is more natural than the standard definition of computable dimension 1.

\begin{thm}\label{dim1isHN}  A real $x$ is highly normal if and only if $\dim_{\comp}(x)=1$.\end{thm}

\begin{proof} 
Assume that $\dim_{\comp}(x)<1$. Then $K_M(x[1..n])< (1-\epsilon)n$ for some always halting Turing machine $M$, $\epsilon>0 $ and almost all $n$. 
Define $f=f_M$ a computable function. Then $x$ is not $f$-weakly normal and therefore $x$ is not highly normal.

For the other direction, suppose that $f$ is a computable function and $x$ is not weakly $f$-normal. Then there exists an $\epsilon>0$ and for infinitely many $n$, we have $C^f_n(x)< (1-\epsilon)n$. So there is a $\sigma$ such that $|\sigma|<(1-\epsilon)n$ where $|f(\sigma)-x|< 2^{-n}$. Hence, for an always halting Turing machine $M$, we have \[K_M(x[1..n])<(1-\epsilon)n+c+2+2\log n,\] where $c$ is the length of a program computing $f$, and we use two extra bits to decide on $f(\sigma)[1..n]$, its successor, or its predecessor. So,
\begin{align*}
    \liminf_{n\to\infty}\frac{K_M(x[1..n])}{n}&<\liminf_{n\to\infty}\frac{(1-\epsilon)n+c+2+2\log n}{n}=1-\epsilon.
\end{align*}
It follows that $\dim_{\comp}(x)<1$.

\end{proof}

\begin{cor}
    There is a highly normal number $x$ which is not supernormal.
\end{cor}
\begin{proof}
$\overline{\Omega}$ is has effective dimension $1$ and so by Theorem \ref{dim1isHN}, it is highly normal and by Proposition \ref{cOmeganotSN}, it is not supernormal.
\end{proof}




The following property of the class of highly normal numbers {\color{blue} directly follows from  Theorem \ref{dim1isHN}.}

\begin{prop}
    Highly normal numbers are closed under complement.
\end{prop}

\bmhead{Acknowledgements}

The authors acknowledge the support of the 5-day workshop 23w5055 held at Banff International Research Station where this project was initiated.  The third author's work was supported in
part by Spanish Ministry of Science and Innovation grants
PID2019-104358RB-I00 and  PID2022-138703OB-I00  and by the Science dept. of Aragon Government:
Group Reference T64$\_$20R (COSMOS research group).
The fifth author's work was partially supported by a Focused Research Group grant from the National Science Foundation of the United States, DMS-1854355.  The authors are grateful to Joseph Miller for helpful conversations about this project.



\end{document}